\documentclass{amsart}

\usepackage{amsmath,amstext,amsfonts}
\usepackage{graphicx}

\begin{document}
\title{Compact differences of composition operators}

\author{Katherine Heller}
\address{Katherine Heller, Department of Mathematics, P. O. Box 400137,
University of Virginia, Charlottesville, VA 22904}
\email{kch3c@virginia.edu}

\author{Barbara D. MacCluer}
\address{Barbara D. MacCluer, Department of Mathematics, P.O. Box 400137,
University of Virginia, Charlottesville, VA 22904}
\email{bdm3f@virginia.edu}

\author{Rachel J. Weir}
\address{Rachel J. Weir, Department of Mathematics,
Allegheny College, Meadville, PA 16335}
\email{rweir@allegheny.edu}

\thanks{The third named author would like to thank the Allegheny College Academic Support Committee
for funding provided during the development of this paper.}


\date{June 8, 2010}

\newtheorem{thm}{Theorem}
\newtheorem{prop}{Proposition}
\newtheorem{lemma}{Lemma}
\newtheorem{cor}{Corollary}

\newtheorem{remark}{\bf Remark}[section]
\newtheorem{example}[remark]{\bf Example}
\newtheorem{definition}[remark]{\bf Definition}

\newcommand{\p}{\varphi}
\newcommand{\cp}{C_{\varphi}}
\newcommand{\D}{{\mathbb D}}
\newcommand{\B}{{\mathcal B}}
\newcommand{\com}{{[C_{\varphi}^*,C_{\varphi}]}}

 \newcommand{\ebox}{\hspace*{1em}\linebreak[0]\hspace*{1em}\hfill\rule[-1ex]{
.1in}{.1in}\\}
 \newcommand{\pf}{\noindent \bf Proof.\rm\ \ }

\begin{abstract}
When $\varphi$ and $\psi$ are linear-fractional self-maps of the unit
ball $B_N$ in ${\mathbb C}^N$, $N\geq 1$, we show that the difference
$C_{\varphi}-C_{\psi}$ cannot be non-trivially compact on either the
Hardy space $H^2(B_N)$ or any weighted Bergman space $A^2_{\alpha}(B_N)$. 
Our arguments emphasize geometrical properties of the inducing maps
$\varphi$ and $\psi$.
\end{abstract}

\maketitle

\section{Introduction}
For a domain $\Omega$ in ${\mathbb C}^N$, where $N\geq 1$, and an analytic
map $\varphi:\Omega\rightarrow \Omega$, we define the composition
operator $C_{\varphi}$ by $C_{\varphi}(f)=f\circ\varphi$, where $f$ is analytic
in $\Omega$.  In the case that $\Omega=\D$, the unit disk in ${\mathbb C}$,
every composition operator acts boundedly on the Hardy space
$$H^2(\D)=\left\{f\mbox{ analytic in }\D:\|f\|_{H^2}^2\equiv\lim_{r\rightarrow 1^-}\int_0^{2\pi}|f(re^{i\theta})|^2\frac{d\theta}{2\pi}<\infty\right\}$$
and weighted Bergman spaces
$$A^2_{\alpha}(\D)=\left\{f\mbox{ analytic in }\D:\|f\|_{\alpha}^2\equiv(\alpha+1)\int_{\D}|f(z)|^2(1-|z|^2)^{\alpha}dA(z)<\infty\right\},$$
where $\alpha>-1$ and $dA$ is normalized area measure.
If $\D$ is replaced by the unit ball $B_N$ in ${\mathbb C}^N$, $N>1$, it is no longer the
case that every composition operator is bounded on the Hardy or weighted Bergman space
of the ball (these spaces are defined in Section~3). However for a large class of maps $\varphi$,
including the rich class of linear-fractional maps, boundedness does continue to hold.

Many properties of composition operators have been studied over the past four decades; the
monographs \cite{cmbook} and \cite{Sbook} give an overview of the work before the mid-1990's.
Recently there has been considerable interest in studying algebras of composition operators,
often modulo the ideal of compact operators (see, for example, \cite{jury1}, \cite{jury2}, \cite{kmm1},
\cite{kmm}, \cite{kmm2}).  In this direction, the question of when a difference
$C_{\varphi}-C_{\psi}$ is compact naturally arises.

The main result of this paper shows that if $\varphi$ and $\psi$ are linear-fractional
self-maps of $B_N$, then $C_{\varphi}-C_{\psi}$ cannot be non-trivially compact; i.e.
if the difference is compact,
either $C_{\varphi}$ and $C_{\psi}$ are individually compact (this happens precisely
when $\|\varphi\|_{\infty}<1$ and $\|\psi\|_{\infty}<1$), or $\varphi=\psi$.
While our focus is on the several variable case,
we begin with a simplified proof of this result in one variable.
The fact that a difference of linear-fractional
composition operators cannot be non-trivially compact on
$H^2(\D)$ or $A^2_{\alpha}(\D)$ was first obtained by
P. Bourdon \cite{bour} and J. Moorhouse \cite{moor} as a consequence of
results on the compactness of a difference of more general composition
operators in one variable.  Our approach here is self-contained, and 
takes a geometric perspective,
which will allow us to generalize our arguments to several variables.
The analogy between the one and several variable arguments is not perfect,
owing to a number of phenomena that are present when $N>1$
but do not occur when $N=1$.  Nevertheless, our geometric approach when
$N=1$ leads us to a tractable way to proceed when $N>1$, and highlights the
new phenomena which must be addressed.
Since our arguments
are essentially the same for either the Hardy or weighted Bergman spaces,
in what follows we will let ${\mathcal H}$ denote any of these spaces.
Our starting point, in either the disk or the ball, will be the following
necessary condition for compactness of $C_{\varphi}-C_{\psi}$.

\begin{thm}\label{phdist}[\,\cite{mac-weir},\cite{moor}\,]
Suppose $\varphi, \psi$ are holomorphic self-maps of $\D$ (respectively, $B_N$) and suppose
that there exists a sequence of points $z_n$ tending to the boundary
of  $\D$ ($B_N$) along which
\begin{equation}\label{phexpress}
\rho(\varphi(z_n), \psi(z_n))\left(\frac{1-|z_n|^2}{1-|\varphi(z_n)|^2}+\frac{1-|z_n|^2}{1-|\psi(z_n)|^2}\right)
\end{equation}
does not converge to zero, where $\rho(\varphi(z_n), \psi(z_n))$ is defined by
\begin{equation}\label{defofrho}
1-(\rho(\varphi(z_n), \psi(z_n)))^2=\frac{(1-|\varphi(z_n)|^2)(1-|\psi(z_n)|^2)}{|1-\langle\varphi(z_n),\psi(z_n)\rangle|^2}.
\end{equation}
Then $C_{\varphi}-C_{\psi}$ is not compact on ${\mathcal H}$.
\end{thm}

For $z$ and $w$ in $\D$ or $B_N$, the quantity $\rho(z,w)$ will be referred to as the pseudohyperbolic distance
between $z$ and $w$,
so that the first factor in~(\ref{phexpress}) is the pseudohyperbolic
distance between $\varphi(z_n)$ and $\psi(z_n)$.  In the disk, the pseudohyperbolic
distance has the simpler expression
\begin{equation}\label{phdisksimp}
\rho_{\D}(z, w)=\left|\frac{z-w}{1-\overline{z}w}\right|.
\end{equation}

 \section{Results in one variable}
 
Throughout this section ${\mathcal H}$ denotes either the Hardy space $H^2(\D)$ or
a weighted Bergman space $A^2_{\alpha}(\D)$, as defined in the previous section.

\begin{thm}\label{mainforD}
Suppose that $\varphi$ and $\psi$ are linear-fractional self-maps of $\D$.
If the difference $C_{\varphi}-C_{\psi}$ is compact on ${\mathcal H}$, then
either $\varphi=\psi$, or both $\|\varphi\|_{\infty}$ and $\|\psi\|_{\infty}$
are strictly less than $1$, so that $C_{\varphi}$ and $C_{\psi}$ are individually compact.
\end{thm}

The key step in our proof of Theorem~\ref{mainforD} is contained in the following
result.

\begin{thm}\label{Dsamefirst}
Suppose $\varphi$ and $\psi$ are non-automorphism linear-fractional self-maps of $\D$ with
$\varphi(\zeta)=\psi(\zeta)\in\partial\D$ and $\varphi'(\zeta)=\psi'(\zeta)$
for some $\zeta\in\partial \D$.
If  $\varphi\neq \psi$ then $C_{\varphi}-C_{\psi}$ is not compact on ${\mathcal H}$.
\end{thm}
\begin{proof}
By pre- and post- composing with rotations, we may assume
without loss of generality that $\zeta =1$ and $\varphi(\zeta)=\psi(\zeta)=1$.
Since $\varphi$ and $\psi$ are linear-fractional, we may also assume without loss
of generality that $\varphi(\D)\subseteq\psi(\D)$, so that $\tau_1\equiv\psi^{-1}\circ\varphi$
is a well-defined linear-fractional self-map of $\D$.  Note that $\tau_1(1)=1$ and $\tau_1'(1)=1$.
Since $\varphi\neq \psi$, $\tau_1$ is not the identity.  Thus $\tau_1$ is conjugate
via the Cayley transform
$$C(z)=i\frac{1+z}{1-z}$$
to a translation $w\rightarrow w+b$ of the upper half-plane ${\mathbb H}=\{w:\mbox{Im }w>0\}$ for some
$b\neq 0$ with $\mbox{Im }b\geq 0$.  
Moreover, it is easy to see that 
$$\psi\circ\tau_1=\tau_2\circ\psi$$ for some
linear-fractional $\tau_2$ which is also conjugate to a translation in the upper half-plane.
Specifically, if $\tau_1$ is conjugate to translation by $b$, then $\tau_2$
is conjugate to translation by $c=b/|\psi'(1)|$; see Lemma 5 in \cite{kmm}.  
Since $b\neq 0$, so also $c\neq 0$.

For any positive number $k$, the line
$\{\mbox{Im }w=k\}$ corresponds under the Cayley transform 
to $E_k\equiv\{z:|1-z|^2=\frac{1}{k}(1-|z|^2)\}$, which is an
internally tangent circle
in $\D$ passing through $1$.  The radius of this circle is equal to $(k+1)^{-1}$.
By choosing $k$ sufficiently large, this circle will be contained in $\psi(\D)\cup \{1\}$.
Fix such a $k$ and choose points $w_n$ on $\{\mbox{Im }w=k\}$ with $w_n\rightarrow\infty$.
The corresponding points $v_n=C^{-1}(w_n)$ in the disk tend to $1$ along the circle
$E_k$, and each $v_n$ is the image under
$\psi$ of some $z_n$ belonging to
the internally tangent circle $\psi^{-1}(E_k)=E_{k'}$.
Notice that $z_n\rightarrow 1$ as $n\rightarrow\infty$.

Next we compute the pseudohyperbolic distance between $\varphi(z_n)$ and $\psi(z_n)$.
To simplify the computations, we define the pseudohyperbolic distance
in the upper half-plane ${\mathbb H}$ by 
$$\rho_{\mathbb H}(u,v)=\rho_{\D}(C^{-1}u,C^{-1}v)$$
for $u$ and $v$ in ${\mathbb H}$.  Using this definition and ~(\ref{phdisksimp}) it is straightforward
to see that
$$\rho_{\mathbb H}(u,v)=\left|\frac{u-v}{u-\overline{v}}\right|.$$
Since $\varphi=\psi\circ\tau_1=\tau_2\circ\psi$,
\begin{eqnarray*}
\rho_{\D}(\varphi(z_n),\psi(z_n))&=&\rho_{\D}(\tau_2(\psi(z_n)),\psi(z_n))=\rho_{\D}(\tau_2(v_n),v_n)\\
&=&\rho_{{\mathbb H}}(C(\tau_2(v_n)),C(v_n))=\rho_{{\mathbb H}}(C(\tau_2(C^{-1}(w_n)),CC^{-1}(w_n))\\&=&
\rho_{{\mathbb H}}(w_n+c,w_n)\\&=&\left|\frac{c}{2i\mbox{Im }w_n+c}\right|=\left|\frac{c}{2ik+c}\right|.
\end{eqnarray*}
Thus for all $n$, the pseudohyperbolic distance between $\varphi(z_n)$ and $\psi(z_n)$ is
a positive constant.

Turning to the second factor in Equation~(\ref{phexpress}), we have
$$\frac{1-|z_n|^2}{1-|\psi(z_n)|^2}=\frac{k'(|1-z_n|^2)}{k(|1-\psi(z_n)|^2)}$$
by the geometry of the sequence $\{z_n\}$ already noted.
Thus since $\psi$ is differentiable at $1$ with $\psi'(1)\neq 0$ and $\psi(1)=1$, 
$$\lim_{n\rightarrow\infty}\frac{1-|z_n|^2}{1-|\psi(z_n)|^2}=\frac{k'}{k|\psi'(1)|^2}\neq 0.$$
Thus we have shown that
$$\rho(\varphi(z_n),\psi(z_n))\left(\frac{1-|z_n|^2}{1-|\psi(z_n)|^2}\right)$$
has a positive limit as $n\rightarrow 1$ (where $z_n\rightarrow 1$). 
Theorem~\ref{phdist} guarantees that $C_{\varphi}-C_{\psi}$ is not compact on ${\mathcal H}$.
\end{proof}

\begin{proof}[Proof of Theorem~\ref{mainforD}.]

If either $\|\varphi\|_{\infty}<1$ or $\|\psi\|_{\infty}<1$, then the compactness
of the difference $C_{\varphi}-C_{\psi}$ implies the compactness of each operator individually.
Thus we may assume $\|\varphi\|_{\infty}=\|\psi\|_{\infty}=1$.  Suppose $\varphi(\zeta)=\eta$
where $\zeta,\eta$ are in $\partial \D$.  Since $\varphi$ and $\psi$ are linear fractional,
both $\varphi'(\zeta)$ and $\psi'(\zeta)$ exist and are non-zero.  If either $\varphi(\zeta)\neq \psi(\zeta)$
or $\varphi'(\zeta)\neq \psi'(\zeta)$, then by Theorem 9.16 of \cite{cmbook} the essential norm
of $C_{\varphi}-C_{\psi}$ satisfies
$$\|C_{\varphi}-C_{\psi}\|_e^2\geq |\varphi'(\zeta)|^{\beta}$$
for some positive number $\beta$ depending on the particular choice of ${\mathcal H}$ in question;
when ${\mathcal H}=H^2(\D)$ we may take $\beta=1$, and when ${\mathcal H}=A^2_{\alpha}(\D)$, $\beta=\alpha+2$.
This gives a positive lower bound on the essential norm of the difference,
so that if the difference is compact we must have $\varphi(\zeta)=\psi(\zeta)$ and
$\varphi'(\zeta)=\psi'(\zeta)$.  Note that this argument also shows that if
$C_{\varphi}-C_{\psi}$ is compact but non-zero, neither $\varphi$ nor
$\psi$ can be an automorphism. An appeal 
to Theorem~\ref{Dsamefirst} finishes the proof.

\end{proof}

\section{Results in several variables}

In this section 
${\mathcal H}$ will denote either the Hardy space $H^2(B_N)$ or a weighted Bergman space $A^2_{\alpha}(B_N)$,
where $B_N$ is 
the ball 
$$\{(z_1,z_2,\cdots,z_N)\in {\mathbb C^N}:\sum_{j=1}^N|z_j|^2<1\}$$
in ${\mathbb C}^N$, $N>1$.
These Hilbert spaces are defined by
$$H^2(B_N)=
\{f\mbox{ analytic on }B_N:\|f\|^2\equiv
\sup_{0<r<1}\int_{\partial B_N}|f(r\zeta)|^2d\sigma(\zeta)<\infty\},
$$
where $\sigma$ is normalized Lebesgue surface area measure on $\partial B_N$,
and for $\alpha>-1$,
$$A^2_{\alpha}(B_N)=\{f\mbox{ analytic on }B_N:\|f\|_{\alpha}^2\equiv
\int_{B_N}|f(z)|^2w_{\alpha}(z)d\nu(z)<\infty\},$$
where 
$$w_{\alpha}(z)=\frac{\Gamma(N+\alpha+1)}{\Gamma(N+1)\Gamma(\alpha+1)}(1-|z|^2)^{\alpha},$$
and $\nu$
is normalized Lebesgue volume measure on $B_N$.
In $H^2(B_N)$ the reproducing kernel for the bounded linear
functional of evaluation at $w\in B_N$ is
\begin{equation}\label{H2ker}
K_w(z)=\frac{1}{(1-\langle z,w\rangle)^N}
\end{equation}
which has norm $(1-|w|^2)^{N/2}$.
The reproducing kernel for evaluation at $w$
in $A^2_{\alpha}(B_N)$ is 
\begin{equation}\label{A2ker}
K_w(z)=\frac{1}{(1-\langle z, w\rangle)^{N+1+\alpha}}.
\end{equation}
 
By a linear-fractional map of $B_N$ we mean a map that is 
analytic in $B_N$ and of the form
\begin{equation}\label{defoflinfrac}
\varphi(z)=\frac{Az+B}{\langle z,C\rangle +d},
\end{equation}
where $A$ is an $N\times N$ matrix, $B$ and $C$ are $N\times 1$ column
vectors, $d$ is a complex scalar, and $\langle \cdot,\cdot \rangle$ is
the usual inner product on ${\mathbb C}^N$.  If $\varphi$ maps
$B_N$ into itself, then necessarily
$|d|>|C|$, so that in particular $d\neq 0$, and $\varphi$ is
analytic in a neighborhood of $\overline{B_N}$.  When
$\varphi$ is a linear-fractional self-map of $B_N$, $C_{\varphi}$ is
bounded on $H^2(B_N)$ and $A^2_{\alpha}(B_N)$ for
$\alpha>-1$ (\cite{cowenmac}).  Moreover, every automorphism of
$B_N$ is linear-fractional (\cite{Ru}, Theorem 2.2.5). 
An important geometric property of linear-fractional maps is that
they take affine sets into affine sets (\cite{cowenmac}, Theorem 7).  By an affine set in ${\mathbb C}^N$ we
mean the translate of a complex subspace; the dimension of the affine set is
the dimension of the subspace.  An ``affine subset of $B_N$'' is the intersection 
of $B_N$ with an affine set in ${\mathbb C}^N$.

The goal of this section
is to obtain the following result.

\begin{thm}\label{scvmain}
If $\varphi$ and $\psi$ are linear-fractional self maps of $B_N$ with
$C_{\varphi}-C_{\psi}$ compact on ${\mathcal H}$, then either
$\varphi=\psi$, or $\|\varphi\|_{\infty}$ and $\|\psi\|_{\infty}$ are both
strictly less than $1$, so that $C_{\varphi}$ and $C_{\psi}$ are compact
on ${\mathcal H}$.
\end{thm}

The key geometric arguments of the last section can be
extended to several variables.  The Cayley upper half space
${\mathbb H}_N$ is defined by
$${\mathbb H}_N=\{(w_1,w'):\mbox{Im }w_1>|w'|^2\}.$$
where $w'=(w_2,..., w_N)$ and $|w'|^2=|w_2|^2+\cdots+|w_N|^2$.
Its boundary is of course
$$\{(w_1,w'):\mbox{Im }w_1=|w'|^2\}.$$
Let $e_1=(1,0,\cdots,0)=(1,0')$.
The (linear-fractional) Cayley transform
$$C(z)=i\frac{e_1+z}{1-z_1}$$
is a biholomorphic map of the ball $B_N$ onto ${\mathbb H}_N$; its inverse
is
$$C^{-1}(w)=\left(\frac{w_1-i}{w_1+i},\frac{2w'}{w_1+i}\right).$$
If $b=(b_1,b')\in\mathbb{C}^N$, an ${\mathbb H}$-translation is a map of the form
\begin{equation}\label{h-translations}
h_b(w_1,w')=(w_1+2i\langle w',b'\rangle+b_1,w'+b').
\end{equation}
If $\mbox{Im }b_1\geq |b'|^2$ it maps  ${\mathbb H}_N$ into itself.
It is an automorphism of ${\mathbb H}_N$ if $\mbox{Im }b_1= |b'|^2$.
The following two facts, which generalize results we used in the
previous section, are easily checked: 
\begin{itemize}
\item An ${\mathbb H}$-translation $h_b$  maps
the set $\Gamma_k\equiv\{(w_1,w'):\mbox{Im }w_1-|w'|^2=k\}$ into the corresponding
set $\{(w_1,w'):\mbox{Im }w_1-|w'|^2=\tilde{k}\}$
where 
$\tilde{k}=k+\mbox{Im }b_1-|b'|^2$.
\item For any $k> 0$, 
$$C^{-1}(\Gamma_k)= E(k,e_1)\equiv\{z\in B_N:|1-z_1|^2= \frac{1}{k}(1-|z|^2)\}.$$
\end{itemize}
The set $E(k,e_1)$ is an internally tangent ellipsoid at $e_1=(1,0')$;
a computation shows that $E(k,e_1)$ consists of the points $(z_1,z')$ satisfying
$$\left|z_1-\frac{k}{1+k}\right|^2+\frac{1}{1+k}|z'|^2=\left(\frac{1}{1+k}\right)^2.$$
In particular, for $t$ real, points of the form $(t+i(k+|w'|^2), w')$ in $\Gamma_k$ pull back under
$C^{-1}$ to points on the ellipsoid $E(k,e_1)$.  For fixed $w'$, these pull-back
points tend to $e_1$ as $t\rightarrow\infty$, and for fixed $t$, they tend to $e_1$
as $|w'|\rightarrow\infty$.

Recall that the pseudohyperbolic metric $\rho_{B_N}(\cdot,\cdot)$ on $B_N$ is defined by
$$1-\rho_{B_N}(z,w)^2=\frac{(1-|z|^2)(1-|w|^2)}{|1-\langle z,w\rangle|^2}.$$
For points $v, u$ in ${\mathbb H}_N$, write $\rho_{{\mathbb H}}(v,u)$ for 
$\rho_{B_N}(C^{-1}v,C^{-1}u)$;  
we will call this the pseudohyperbolic metric on ${\mathbb H}_N$.
Since the pseudohyperbolic
metric $\rho_{B_N}$ is easily seen to be automorphism invariant, it follows that
for any automorphism $\Lambda$ of ${\mathbb H}_N$,
$$\rho_{{\mathbb H}}(\Lambda v,\Lambda u)=\rho_{{\mathbb H}}(v,u).$$
In the next theorem, we will use this observation  
with $\Lambda$
an automorphic ${\mathbb H}$-translation.

By a parabolic linear-fractional
map in $B_N$ fixing $e_1$ we mean a linear-fractional map $\tau$ of $B_N$ into $B_N$ with
$\tau(e_1)=e_1$ and $D_1\tau_1(e_1)=1$, but fixing no other point
in $\overline{B_N}$.
By \cite{braccietal} any parabolic linear-fractional self-map $\varphi$ of $B_N$ that fixes $e_1$
is conjugate to a self-map of ${\mathbb H}_N$ of the
form
$$\Phi(w_1,w')=(w_1+2i\langle w', \delta\rangle +b,Aw'+\gamma)$$
(where $\delta$ and $\gamma$ are in ${\mathbb C}^{N-1}$, $b\in {\mathbb C}$,
and certain conditions hold, including 
$|A|\leq 1$).  Note that the ${\mathbb H}$-translations of Equation~(\ref{h-translations})
are a special case of this family of maps.
Conjugating $\Phi$ by the Cayley transform $C$ we see that the first coordinate function
of $C^{-1}\Phi C$ is
\begin{equation}\label{firstcoordpara}
\frac{(2i-b)z_1-2\langle z',\delta\rangle +b}
{-bz_1-2\langle z',\delta\rangle+2i+b}.
\end{equation}
We will need this explicit formula in the proof of the next result.

\begin{thm}\label{parabolics}
Suppose $\varphi$ and $\psi$ are parabolic linear-fractional self-maps of
the ball fixing $e_1$, so that 
$$D_1\varphi_1(e_1)=1=D_1\psi_1(e_1).$$  If $\varphi\neq\psi$, then
$C_{\varphi}-C_{\psi}$ is not compact on ${\mathcal H}$.
\end{thm}
\begin{proof}

We will show that for distinct maps $\varphi$ and $\psi$ as in the hypothesis, there exists
a sequence of points $\{z^{(n)}\}$ in $B_N$ such that
\begin{itemize}
\item[(a)] $z^{(n)}\rightarrow e_1$ as $n\rightarrow \infty$.
\item[(b)] For all $n$, $\rho_{B_N}(\varphi(z^{(n)}),\psi(z^{(n)}))$ has a strictly positive constant value.
\item[(c)] $(1-|\varphi(z^{(n)})|^2)/(1-|z^{(n)}|^2)$ has finite positive limit as $n\rightarrow \infty$.
\end{itemize}
An appeal to Theorem~\ref{phdist} will then complete the proof.

We will use the corresponding upper case letters for a self-map of $B_N$ conjugated to ${\mathbb H}_N$,
so that $\Phi=C\varphi C^{-1}$ and $\Psi=C\psi C^{-1}$.  These maps have the forms
$$\Phi(w_1,w')=(w_1+2i\langle w',\delta_1\rangle+b_1,A_1w'+\gamma_1) $$
and
$$\Psi(w_1,w')=(w_1+2i\langle w',\delta_2\rangle+b_2,A_2w'+\gamma_2) $$
for some $\delta_i$ and $\gamma_i$ in ${\mathbb C}^{N-1}$, scalars $b_i$
and $(N-1)\times (N-1)$ matrices $A_i$, $i=1,2$.

Fix a sequence of points $\{w^{(n)}\}= \{(w_1^{(n)},c)\}$ in ${\mathbb H}_N$,
where $c$ is a constant in ${\mathbb C}^{N-1}$ to be determined,
satisfying
\begin{itemize}
\item[(i)] $\mbox{Im }w_1^{(n)}-|c|^2=k$
\item[(ii)] $C^{-1}(w_1^{(n)},c)\rightarrow e_1$
\end{itemize}
where $k>0$ is fixed but arbitrary.
By (i), the points $C^{-1}(w^{(n)})$ lie on the
ellipsoid
$E(k,e_1)$.
We can ensure that condition (ii) holds by requiring that $\mbox{Re}\ w_1^{(n)}\rightarrow
\infty$. 

Let $$P_1^{(n)}=\Phi(w^{(n)})=(w_1^{(n)}+2i\langle c,\delta_1\rangle+b_1,A_1c+\gamma_1)
$$
and
$$P_2^{(n)}=\Psi(w^{(n)})=(w_1^{(n)}+2i\langle c,\delta_2\rangle+b_2,A_2c+\gamma_2).
$$  
Since the pseudohyperbolic metric is automorphism invariant, we have
$$
\rho_{\mathbb H}(P_1^{(n)},P_2^{(n)})=\rho_{\mathbb H}(h(P_1^{(n)}),
h(P_2^{(n)}))$$
where $h$ is the automorphic ${\mathbb H}$-translation given by
$$h(v_1,v')=(v_1-\mbox{Re } w_1^{(n)},v').$$
Thus for any positive integer $n$,
the distance $\rho_{\mathbb H}(P_1^{(n)},P_2^{(n)})$ is equal to
$$\rho_{\mathbb H}((i(k+|c|^2)+2i\langle c,\delta_1\rangle+b_1,A_1c+\gamma_1),
(i(k+|c|^2)+2i\langle c,\delta_2\rangle+b_2,A_2c+\gamma_2)).
$$
Note that this quantity is independent of the particular point
$w^{(n)}$ in our sequence chosen to satisfy of
(i) and (ii), and that if $\varphi\neq \psi$
(so that not all of $b_1=b_2, \delta_1=\delta_2, A_1=A_2$ and $\gamma_1=\gamma_2$
hold) we may certainly choose $c$ so that this quantity is not $0$.
Thus for such a choice, condition (ii) gives the existence of a sequence of
points $\{z^{(n)}\}$ in $B_N$ tending to $e_1$ along $E(k,e_1)$ for which
$$\rho_{B_N}(\varphi(z^{(n)}),\psi(z^{(n)}))$$ is a positive constant value;
the $z^{(n)}$'s being just the inverse images under the Cayley transform
$C$ of our chosen
points $w^{(n)}$ in ${\mathbb H}_N$.  Hence conditions (a) and (b) hold.

For property (c), first note that the images under $\Phi$ of points of the form
$(w_1,c)$ with $\mbox{Im }w_1-|c|^2=k$  look like
$$(w_1+2i\langle c,\delta_1\rangle+b_1,A_1c+\gamma_1),$$
and for these points we see that
$$\mbox{Im }(w_1+2i\langle c,\delta_1\rangle+b_1)-|A_1c+\gamma_1|^2=k+|c|^2+
\mbox{Im}(2i\langle c,\delta_1\rangle+b_1)-|A_1c+\gamma_1|^2,$$
which is constant, say $k'$.  In other words, the image under $\varphi$ of our points
$z^{(n)}$ lie on some ellipsoid $E(k',e_1)$
and
$$|1-\varphi_1(z^{(n)})|^2= \frac{1}{k'}(1-|\varphi(z^{(n)})|^2).$$
Moreover, since the points $z^{(n)}$ lie on $E(k,e_1)$, we have
$$|1-z^{(n)}_1|^2= \frac{1}{k}(1-|z^{(n)}|^2).$$
Thus
\begin{equation}\label{relatingtoder}
\frac{1-|\varphi(z^{(n)})|^2}{1-|z^{(n)}|^2}=\frac{k'}{k}\frac{|1-\varphi_1(z^{(n)})|^2}{|1-z^{(n)}_1|^2}.
\end{equation}

Since $\varphi_1$ is differentiable at $e_1$, 
we have a Taylor series expansion of $\varphi_1$ in a neighborhood of
$e_1$:
\begin{eqnarray*}
\varphi_1(z)&=&\varphi_1(e_1)+D_1\varphi_1(e_1)(z_1-1)+\sum_{j=2}^{N}D_j\varphi_1(e_1)z_j+
\frac{1}{2!}D_{11}\varphi_1(e_1)(z_1-1)^2\\&+&\sum_{j=2}^ND_{1j}\varphi_1(e_1)(z_1-1)z_j+\sum_{k,j=2;k\neq j}^ND_{kj}\varphi_1(e_1)z_kz_j
+\frac{1}{2!}\sum_{j=2}^ND_{jj}\varphi_1(e_1)z_j^2+\cdots.
\end{eqnarray*}
Recall that by
hypothesis $D_1\varphi_1(e_1)=1$.  Direct computation using Equation~(\ref{firstcoordpara}) shows that $D_j\varphi_1(e_1)=0,$ 
for $j=2,...,N$ (this also follows more generally from the fact that
$e_1$ is a fixed point of $\varphi$; see Lemma 6.6 in \cite{cmbook}) and $D_{kj}\varphi_1(e_1)=0$ for $k,j=2,...,N$. 
Thus
$$\varphi_1(z)-1=(z_1-1)+\frac{1}{2!}\left[D_{11}\varphi_1(e_1)(z_1-1)^2
+2\sum_{j=2}^ND_{1j}\varphi_1(e_1)(z_1-1)z_j\right]
+\cdots,$$
where the $+\cdots$ indicates higher order terms
of the form
$$\frac{D^{\nu}\varphi_1(e_1)(z-e_1)^{\nu}}{\nu!},$$
and $\nu$ is a multi-index of order at least $3$.  Since
for $z\in B_N$ we have
$$\frac{|z_2|^2+\cdots+|z_N|^2}{|1-z_1|}\leq \frac{1-|z_1|^2}{|1-z_1|}\leq 2\frac{1-|z_1|}{|1-z_1|}\leq 2$$
we see that
$$\frac{1-\varphi_1(z)}{1-z_1}\rightarrow 1 \mbox{ as }z\rightarrow e_1 \mbox{ in }B_N.$$
By~(\ref{relatingtoder}) this implies that
$$\frac{1-|\varphi(z^{(n)})|^2}{1-|z^{(n)}|^2}$$ has a positive finite limit as $n\rightarrow \infty$, 
and property (c) holds as desired.

\end{proof}

To prove Theorem~\ref{scvmain} we will use the preceding result and
the following qualitative generalization of Theorem 9.16 in \cite{cmbook},
specialized to linear-fractional maps.  In the statement we use
the notation $\psi_{\zeta}$ for the coordinate of $\psi$ in the
$\zeta-$direction, that is $\psi_{\zeta}(z)=\langle \psi(z),\zeta\rangle$.
Moreover the derivative of $\psi_{\zeta}$ in the $\eta$ direction,
denoted $D_{\eta}\psi_{\zeta}$, 
is defined by $D_{\eta}\psi_{\zeta}(z)\equiv \langle \psi'(z)\eta,\zeta\rangle$.
Note that when $\zeta=\eta=e_1$ this is just $D_1\psi_1(z)$.

For $\eta\in\partial B_N$, write $[\eta]$ for the complex
line containing $\eta$ and $0$; that is, the one-dimensional
subspace of $\overline{B_N}$ consisting of
all points of the form $\{\alpha \eta:\alpha \in \mathbb{C}\}$.
In particular, the complex line $[e_1]$ intersected with $B_N$
consists of all points in the ball
whose last $N-1$ coordinates are $0$.

\begin{thm}\label{samefirst}
Suppose $\varphi$ and $\psi$ are linear-fractional self-maps of $B_N$, 
and suppose $\varphi(\zeta)=\zeta$ for some
$\zeta\in\partial B_N$.
If either $\psi(\zeta)\neq \zeta$, or $\psi(\zeta)=\zeta$ and $D_{\zeta}\varphi_{\zeta}(\zeta)\neq
D_{\zeta}\psi_{\zeta}(\zeta)$, then $C_{\varphi}-C_{\psi}$ is not compact.
\end{thm}

\begin{proof}
The argument follows that of Theorem 9.16 in \cite{cmbook}.
Without loss of generality we may assume $\zeta=e_1$.

First suppose $\psi(e_1)\neq e_1$.  If we can find
a sequence of
points $z^{(n)}$ tending to $\partial B_N$, so that 
$$\left\|(C_{\varphi}-C_{\psi})^*\left(\frac{K_{z^{(n)}}}{\|K_{z^{(n)}}\|}\right)\right\|$$
is bounded away from $0$,
where $K_{z^{(n)}}$ denotes the kernel function in ${\mathcal H}$ at $z^{(n)}$ (see Equations~(\ref{H2ker})
and~(\ref{A2ker})),
then $C_{\varphi}-C_{\psi}$ is not compact,
since the normalized kernel functions $K_{z^{(n)}}/\|K_{z^{(n)}}\|$ tend weakly to $0$ as $z^{(n)}\rightarrow
\partial B_N$.
Using the fact that for any bounded composition operator $C_{\tau}$ we have
$C_{\tau}^*(K_z)=K_{\tau(z)}$, we see that
$$\|(C_{\varphi}-C_{\psi})^*(K_z)\|^2=\|K_{\varphi(z)}\|^2+\|K_{\psi(z)}\|^2-2\mbox{Re }K_{\varphi(z)}(\psi(z)),$$
and thus
\begin{equation}\label{adjkercomp}
\|(C_{\varphi}-C_{\psi})^*(K_z/\|K_z\|)\|^2\geq \left(\frac{1-|z|^2}{1-|\varphi(z)|^2}\right)^{\beta}
-2\mbox{Re }\frac{K_{\varphi(z)}(\psi(z))}{\|K_z\|^2}
\end{equation}
where $\beta=N$ in $H^2(B_N)$ and $\beta =N+1+\alpha $ in $A^2_{\alpha}(B_N)$.  
With our assumption that $\psi(e_1)\neq e_1$, it is easy to see that as $z\rightarrow e_1$ radially, the second
term on the right hand side of Equation~(\ref{adjkercomp}) tends to $0$.  By Julia-Caratheodory
theory (see for example, \cite{Ru}, Section 8.5), the first term
tends to the positive value $(D_1\varphi_1(e_1))^{-\beta}$.
This shows that $C_{\varphi}-C_{\psi}$ is not compact if $ \psi(e_1)\neq \varphi(e_1)$.

Now suppose $ \psi(e_1)= \varphi(e_1)=e_1$ but $D_1\psi_1(e_1)\neq D_1\varphi_1(e_1)$.
As before, if we can find a sequence of points $z^{(n)}$ in $B_N$ tending to $e_1$
along which 
$$\|(C_{\varphi}-C_{\psi})^*(K_{z^{(n)}}/\|K_{z^{(n)}}\|)\|$$ 
is bounded away from $0$, then we can conclude that $C_{\varphi}-C_{\psi}$ is not
compact.
The sequence $\{z^{(n)}\}$ will chosen so that $z^{(n)}=(z_1^{(n)},0')$ where
$$\frac{|1-z_1^{(n)}|}{1-|z^{(n)}|^2}=\frac{|1-z_1^{(n)}|}{1-|z_1^{(n)}|^2}=M$$
for a fixed and suitably large value of $M$; that is, the points $z^{(n)}$
will approach $e_1$ along the boundary of a non-tangential approach region, of large aperture,
in the complex line $[e_1]$.  To analyze the second term on the right hand side
of Equation~(\ref{adjkercomp}), we first consider
\begin{equation}\label{threeterm}
\frac{1-\langle \varphi(z),\psi(z)\rangle}{1-|z|^2}=
\frac{1-|\varphi(z)|^2}{1-|z|^2}+\frac{\langle \varphi(z)-e_1,\varphi(z)-\psi(z)\rangle}{1-|z|^2}
+\frac{\langle e_1,\varphi(z)-\psi(z)\rangle}{1-|z|^2}.
\end{equation}
The third term on the right hand side of Equation~(\ref{threeterm}) has modulus
$$\frac{|\varphi_1(z)-\psi_1(z)|}{1-|z|^2}=
\left|\frac{1-\psi_1(z)}{1-z_1}-\frac{1-\varphi_1(z)}{1-z_1}\right|
\frac{|1-z_1|}{1-|z|^2},$$
and if $z$ is chosen to approach $e_1$ in $[e_1]$ along
the curve $|1-z_1|/(1-|z_1|^2)=M$ this will tend to
$|D_1\psi_1(e_1)-D_1\varphi_1(e_1)|M$. Since $D_1\varphi_1(e_1)\neq D_1\psi_1(e_1)$
by assumption, this can be made
as large as desired by choosing $M$ large.  

The first term on
the right hand side of Equation~(\ref{threeterm}) tends
to $|D_1\varphi_1(e_1)|$ along any sequence of points approaching
$e_1$ non-tangentially in $[e_1]$.  We claim that the second
term in~(\ref{threeterm}) tends to $0$ along any such sequence.  To see this, it's
enough to show that
$$\frac{|\varphi(z)-e_1||\varphi(z)-\psi(z)|}{1-|z|^2}$$
tends to $0$ as $z$ approaches $e_1$ non-tangentially in $[e_1]$.
Since
$$|\varphi(z)-\psi(z)|\leq |\varphi(z)-e_1|+|e_1-\psi(z)|$$
it suffices to show
that
$$\frac{|\varphi(z)-e_1|^2}{1-|z|^2}$$
and 
$$\frac{|\varphi(z)-e_1||\psi(z)-e_1|}{1-|z|^2}$$
both tend to $0$.  We have
$$\frac{|\varphi(z)-e_1|^2}{1-|z|^2}=\frac{|\varphi_1(z)-1|^2+|\varphi'(z)|^2}{|1-z_1|}\ \frac{|1-z_1|}{1-|z|^2}$$
where $\varphi'(z)$ denotes the $(N-1)$-tuple $(\varphi_2(z),\ldots,\varphi_N(z))$.
We're considering points $z^{(n)}=(z_1^{(n)},0')$ tending to $e_1$ for which
$|1-z_1^{(n)}|/(1-|z_1^{(n)}|^2)$ is some constant value $M$.  Along such a sequence,
$|1-\varphi_1(z)|/|1-z_1|$ tends to $|D_1\varphi_1(e_1)|$, so
that 
$$\frac{|\varphi_1(z)-1|^2}{|1-z_1|}\frac{|1-z_1|}{1-|z|^2}\rightarrow 0.$$
By the Julia-Caratheodory theorem (\cite{Ru}, Theorem 8.5.6),
we also have for $2\leq k\leq N$,
$$\frac{|\varphi_k(z)|^2}{|1-z_1|}\rightarrow 0$$
along any non-tangential sequence approaching $e_1$ in $[e_1]$.
Since a similar analysis applies to show that
$$\frac{|\psi(z)-e_1|^2}{1-|z|^2}\rightarrow 0$$
along the sequences under consideration,
it follows that
$$\frac{|\varphi(z)-e_1||\psi(z)-e_1|}{1-|z|^2}\rightarrow 0$$
as desired.

Thus we have shown the following:  If $\varphi(e_1)=\psi(e_1)=e_1$ but
$D_1\varphi_1(e_1)\neq D_1\psi_1(e_1)$,
then given $\epsilon>0$ there exists $M<\infty$ so that if
$z^{(n)}=(z_1^{(n)},0')$ approaches $e_1$ as $n\rightarrow\infty$, where
$$\frac{|1-z_1^{(n)}|}{1-|z_1^{(n)}|^2}=M,$$
then 
$$\limsup_{n\rightarrow\infty}\frac{|K_{\varphi(z^{(n)})}(\psi(z^{(n)}))|}{\|K_{z^{(n)}}\|^2}<\epsilon.$$
By Equation~(\ref{adjkercomp}) this
says $C_{\varphi}-C_{\psi}$ is not compact.

\end{proof}

Remark:  It's clear that a version of Theorem~\ref{samefirst} holds, with essentially
the same proof, when $\varphi$ and $\psi$ are more general analytic self-maps
of $B_N$, if in the statement of the theorem the values of $\varphi,\psi, D_{\zeta}\varphi_{\zeta}$
and $D_{\zeta}\psi_{\zeta}$ at $\zeta$ are replaced by their radial limits there.
Also a version of the result, with the hypothesis $\varphi(\zeta)=\zeta$ replaced
by $\varphi(\zeta)=\eta$ for $\zeta,\eta\in\partial B_N$ can be formulated.  Since
we do not need these more general versions, we leave the precise statements
to the interested reader.

\begin{prop}\label{niceform2}
Suppose that $\tau$ is a linear-fractional self-map
of $B_N$ such that the restriction of $\tau$ to the 
the complex
line $[e_1]$ in $\overline{B_N}$ is the identity on $[e_1]\cap \partial B_N$.
Then $\tau(z_1,z_2,\cdots, z_N)=(z_1,A'z')$, where
$z'$ denotes $(z_2, \cdots, z_N)$ and $A'$ is an $(N-1)\times(N-1)$
matrix.  
\end{prop}
\begin{proof}
Since $\varphi$ is linear-fractional we have 
$$\varphi(z)=\frac{Az+B}{c_1z_1+c_2z_2+\cdots c_Nz_N+1}$$
for some $N\times N$ matrix $A=(a_{jk})$, $N\times 1$ matrix $B$,
and constants $c_k$.  By hypothesis we must have
$$a_{k1}\lambda+b_k=0$$
for $2\leq k\leq N$ and all $\lambda \in {\mathbb C}$ with $|\lambda|=1$.  Thus
$a_{k1}=b_k=0$ for $2\leq k \leq N$.  From $\varphi(e_1)=e_1$ and
$\varphi(-e_1)=-e_1$ we see that
$b_1=c_1$ and $a_{11}=1$.  Using this, and the requirement
that $a_{11}\lambda +b_1=\lambda(c_1\lambda+1)$ for 
$|\lambda|=1$, we must have $b_1=0$ and thus also $c_1=0$.

Moreover, since $\varphi$ fixes $e_1$, we have by Lemma 6.6 of \cite{cmbook} that
\begin{equation}\label{der1}
D_k\varphi_1(e_1)=0\mbox{ for }k=2,3,\cdots,N. 
\end{equation} 
Since $\varphi(-e_1)=-e_1$, we may apply the same lemma to $-\varphi(-z)$ to
see that 
\begin{equation}\label{der2}
D_k\varphi_1(-e_1)=0 \mbox{ for }k=2,3,\cdots,N.
\end{equation}
Equations~(\ref{der1}) and~(\ref{der2}) together say
$$a_{1k}=c_k=0\mbox{ for }k=2,\cdots,N.$$
This completes the proof.
\end{proof}

To move from Theorem~\ref{parabolics}, which deals with parabolic maps,
to the full result of Theorem~\ref{scvmain}, we will need the notion
of the Krein adjoint of the linear-fractional map $\varphi$.  If $\varphi$
is as
given in Equation~(\ref{defoflinfrac}), its Krein adjoint is defined to be the linear
fractional map
\begin{equation}\label{kreinadjdef}
\sigma_{\varphi}(z)=\frac{A^*z-C}{\langle z,-B\rangle +\overline{d}}.
\end{equation}
This will be a self-map of $B_N$ whenever $\varphi$ is, and
when $\varphi$ is an automorphism, its Krein adjoint is equal
to $\varphi^{-1}$.  Moreover, $\varphi$ and $\sigma_{\varphi}$ have
the same fixed  points on $\partial B_N$; for these and other basic facts, see \cite{cowenmac}
and \cite{mac-weir}.
Properties of the map $\varphi\circ\sigma_{\varphi}$ appear in the next
result.

\begin{thm}\label{univalentcase}
Suppose $\varphi$ and $\psi$ are linear fractional maps with
$\|\varphi\|_{\infty}=\|\psi\|_{\infty}=1$.  Assume further that
at least one of the maps $\varphi, \psi$ is univalent.
If 
$C_{\varphi}-C_{\psi}$ is compact on ${\mathcal H}$,
then $\varphi=\psi$.
\end{thm}
\begin{proof}
By the symmetric roles of $\varphi$ and $\psi$ we may assume
that $\varphi$ is univalent.
The hypothesis $\|\varphi\|_{\infty}=1$ implies that there exists $\zeta$ 
in $\partial B_N$ with $|\varphi(\zeta)|=1$.
Composing on the left and right
by unitaries, there is no loss of generality in assuming $\zeta=e_1$ and $\varphi(e_1)=e_1$.
By Theorem~\ref{samefirst}, we must have
$\psi(e_1)=e_1$ as well.  

Let $\sigma_{\varphi}$ be the 
Krein adjoint of $\varphi$ as defined in Equation~(\ref{kreinadjdef}).  Since $C_{\varphi}-C_{\psi}$ is compact
and $C_{\sigma_{\varphi}}$ is bounded, 
$$C_{\sigma_{\varphi}}(C_{\varphi}-C_{\psi})=C_{\varphi\circ\sigma_{\varphi}}-C_{\psi\circ\sigma_{\varphi}}$$
is also compact.  Set
$\tau=\varphi\circ\sigma_{\varphi}$ and $\xi=\psi\circ\sigma_{\varphi}$.
We have
$$\tau(e_1)=\varphi\circ\sigma_{\varphi}(e_1)=e_1$$
and thus by Theorem~\ref{samefirst}, $\xi(e_1)=e_1$ and $D_1\tau_1(e_1)=D_1\xi_1(e_1)$.
A computation shows that $D_1\tau_1(e_1)=1$, (details of this computation can
be found in the proof of Theorem 2 in \cite{mac-weir}) so that
$$D_1\xi_1(e_1)=1.$$

We claim that $\tau=\xi$.  To see this, note that it is immediate by Theorem~\ref{parabolics}
if neither $\tau$ nor $\xi$ have any fixed point in the open ball $B_N$.
Suppose next that $\tau$
has a fixed point in the open ball
and lying in the complex line $[e_1]$.
Restricting $\tau$ to the intersection of $[e_1]$ and the ball,
we see that $\tau$ must be the identity on
$[e_1]\cap \overline{B_N}$, since $D_1\tau_1(e_1)=1$ (see, for example,
Problem 2.38 in \cite{cmbook}, p. 60).
By Proposition~\ref{niceform2} we see that $\tau$ has
the form $\tau(z_1,z')=(z_1,Az')$ for some $(N-1)\times(N-1)$ matrix
$A$. We label the entries of $A$ as $a_{jk}$ for $j,k=2,\cdots, N$.  
Since $C_{\tau}-C_{\xi}$ is compact, we appeal to Theorem~\ref{samefirst}
to see that, since $\tau$ is the identity at each point of $[e_1]\cap\partial B_N$,
so is $\xi$.  Applying Proposition~\ref{niceform2} again, we see that $\xi(z_1,z')=(z_1,Mz')$
for an $(N-1)\times(N-1)$ matrix $M=(m_{jk})$, $j,k=2,\cdots, n$.
Our goal is to show that $A=M$.

Fix $j$ with $2\leq j\leq N$, and consider the pseudohyperbolic
distance $\rho(\tau(\omega_t),\xi(\omega_t))$ at points of the form
\begin{equation}\label{path}
\omega_t=(t,0,\cdots,\sqrt{1-t},0\cdots,0)=(t,0',\sqrt{1-t},0''),
\end{equation}
for $0<t<1$, where the $\sqrt{1-t}$ appears in the $j^{th}$ component.
These points lie in the ball $B_N$.
A computation shows that
$$1-\rho^2(\tau(\omega_t),\xi(\omega_t))$$ is equal to
$$\frac{[1-t^2-(1-t)\sum_{k=2}^N|a_{kj}|^2][1-t^2-(1-t)\sum_{k=2}^N|m_{kj}|^2]}
{|1-t^2-(1-t)\sum_{k=2}^Na_{kj}\overline{m_{kj}}|^2},
$$ and further computation shows that as $t\uparrow 1$ this has limit
equal to
\begin{equation}\label{lhopital}
\frac{(2-\lambda^2)(2-\gamma^2)}{(2-\sum_ka_{kj}\overline{m_{kj}})(2-\sum_k\overline{a_{kj}}m_{kj})}
\end{equation}
where 
$$\lambda=\left(\sum_{k=2}^N|a_{kj}|^2\right)^{1/2}$$ and 
$$\gamma=\left(\sum_{k=2}^N|m_{kj}|^2\right)^{1/2}.$$ 
Observe that $\lambda$ and $\gamma$ are at most $1$, since
$\tau$ and $\xi$ map the ball into itself.
Write $Z=(a_{2j},a_{3j}\cdots,a_{Nj})$ and $W=(m_{2j},m_{3j},\cdots,m_{Nj}),$
so that $\lambda=\|Z\|$ and $\gamma=\|W\|$.
Moreover, the denominator in~(\ref{lhopital}) is
$|2-\langle Z,W\rangle|^2$,
and by the Cauchy-Schwarz inequality,
$|\langle Z,W\rangle|\leq \lambda\gamma$ with equality only if either $Z=c W$
for some $c\in {\mathbb C}$ or one of $Z,W$ is $0$.

We investigate the condition under
which the expression in~(\ref{lhopital}) is equal to $1$.
We have
$$\frac{(2-\lambda^2)(2-\gamma^2)}{|2-\langle Z,W\rangle|^2}
\leq \frac{(2-\lambda^2)(2-\gamma^2)}{(2-|\langle Z,W\rangle|)^2}
\leq \frac{(2-\lambda^2)(2-\gamma^2)}{(2-\lambda\gamma)^2}\leq 1,$$
with the last inequality following from its equivalence to
$(\lambda-\gamma)^2\geq 0$.
Thus if the expression in~(\ref{lhopital}) is equal to $1$, 
we must have $\lambda=\gamma$ and 
$$|2-\langle Z,W\rangle|=2-|\langle Z,W\rangle|=2-\|Z\|\|W\|.$$
Together these force $Z=W$, which says that the 
 $(j-1)^{st}$ column of $A$ is the
same as the $(j-1)^{st}$ column of $M$.
Thus if $A\neq M$, $\rho(\tau,\xi)$ has a strictly positive limit along some path
as in Equation~(\ref{path}).  
The ratio
$$\frac{1-|z|^2}{1-|\tau(z)|^2}$$
has the positive limit $(2-\lambda^2)^{-1}$ along the same path. 
Applying Theorem~\ref{phdist} we have
a contradiction to the hypothesis that $C_{\tau}-C_{\xi}$ is compact.
Thus $A=M$,
verifying the claim under the assumption that $\tau$ has a fixed point
in $[e_1]\cap B_N$.

Finally, suppose $\tau$ has a fixed point in the intersection of the open ball
and the complex line through $\eta$ and $e_1$ for some $\eta\in\partial B_N$, but
not in $[e_1]$.  Since the automorphisms act doubly transitively on
$\partial B_N$, we may find an automorphism $\Phi$ of the ball,
fixing $e_1$ so that $\widetilde{\tau}\equiv\Phi^{-1}\tau\,\Phi$
fixes $e_1$ and a point of $[e_1]\cap B_N$. 
A computation shows that $D_1\widetilde{\tau}_1(e_1)=1$; this
computation is aided by the 
fact that 
$$D_1\widetilde{\tau}_1(e_1)=\langle \widetilde{\tau}\,'(e_1)e_1,e_1\rangle$$
and the observation that since
$\tau, \Phi$ and $\Phi^{-1}$ all fix $e_1$, we have
$$D_k\tau_1(e_1)=0, D_k\Phi_1(e_1)=0, D_k\Phi_1^{-1}(e_1)=0\mbox{ for all }k=2,3,\ldots,N$$
(\cite{cmbook}, Lemma 6.6). 
Conjugating $\xi$ by $\Phi$ as well to get
$\widetilde{\xi}\equiv\Phi^{-1}\xi\Phi$,
we apply the previous argument to see that $\widetilde{\tau}=\widetilde{\xi}$,
and hence $\tau=\xi$, in this case as well.

Thus compactness of $C_{\varphi}-C_{\psi}$ implies that 
$\tau=\xi,$
or equivalently
\begin{equation}\label{tauxi}
\varphi\circ\sigma_{\varphi}=\psi\circ\sigma_{\varphi}
\end{equation}
on $B_N$, where
$\sigma_{\varphi}$ is the Krein adjoint of $\varphi$.
From this we see that $\varphi$ and $\psi$ agree on the range of $\sigma$.
Since we are assuming that $\varphi$ is univalent, so is $\sigma_{\varphi}$ (\cite{cowenmac})
and it follows that $\varphi=\psi$, since the
range of $\sigma_{\varphi}$ is an open set in $B_N$. 
\end{proof}

The final step is to remove the hypothesis of univalence in the last result
to obtain the full proof of Theorem~\ref{scvmain}, which we turn to next.
It will be helpful to recast our Hilbert space ${\mathcal H}$
as weighted Hardy spaces, defined below, and consider
restriction and extension
operators on these weighted Hardy spaces. 

If $f$ is analytic in
$B_N$, then $f$ has a homogeneous
expansion
\[ f= \sum_s f_s, \]
where, for each $z \in B_N$, we have
\begin{equation}\label{homogeneousfnc}
 f_s(z) = \sum_{|\alpha|=s} c_{\alpha} z^{\alpha}. 
 \end{equation}
Here, $\alpha = (\alpha_1, \dots, \alpha_N)$ and $|\alpha| = \alpha_1 + \cdots +
\alpha_N$. The Hardy space $H^2(B_N)$ is equivalently
defined as 
\begin{equation}\label{h2altdef}
\{f\mbox{ analytic in }B_N:\sum_{s=0}^{\infty}\|f_s\|_2^2<\infty\},
\end{equation}
where $\| \cdot \|_{2}$ is the norm in $L^2(\sigma)$.
The sum in~(\ref{h2altdef}) is $\|f\|_{H^2(B_N)}^2$.
More generally, given a suitable sequence of
positive numbers $\{ \beta(s) \}$, the weighted Hardy space
$H^2(\beta, B_N)$ is the set of
functions $f$ which are analytic in $B_N$ and which satisfy
\[ \|f\|_{H^2(\beta, B_N)}^2\equiv \sum_{s=0}^{\infty} \|f_s\|_{2}^2\  \beta(s)^2 <
\infty. \]
 Note that,
since the monomials
$z^{\alpha}$ are orthogonal on $L^2(\sigma)$ (\cite{Ru}, Section 1.4),
\[ \|f_s\|_{2}^2 =  \sum_{|\alpha|=s} |c_{\alpha}|^2 \| z^{\alpha} \|^2_{2}
= \sum_{|\alpha|=s}
|c_{\alpha}|^2 \frac{(N-1)! \alpha!}{(N-1+s)!} . \]

The next result realizes the weighted Bergman spaces $A^2_{\gamma}(B_N)$ as weighted
Hardy spaces.

\begin{lemma} \label{lem:equiv1}
\begin{itemize}
\item[(a)]
Let $\gamma>-1$ and set $\beta(s)^2 = (s+1)^{-(\gamma+1)}$. We have
$H^2(\beta,B_N)=A^2_{\gamma}(B_N)$, with equivalent norms.
\item[(b)] Let $K $ be an integer with $1 \leq K < N$, and let 
\[ \beta(s)^2=\frac{(N-1)!(K-1+s)!}{(K-1)!(N-1+s)!}(s+1)^{-(\gamma+1)}  \]
where $\gamma\geq -1$.
Then $H^2(\beta,B_K)=A^2_{N-K+\gamma}(B_K)$, with equivalent norms.
\end{itemize}
\end{lemma}

\begin{proof}
We prove (a) first.  Let $f$ be analytic in $B_K$ with homogeneous expansion $f= \sum f_s$, where
$f_s$ is as in Equation~(\ref{homogeneousfnc}).
If $\beta(s)^2=(s+1)^{-(\gamma+1)}$, we have
\[ \|f\|_{H^2(\beta, B_N)}^2 = \sum_{s=0}^{\infty} \sum_{|\alpha|=s}
|c_{\alpha}|^2 \frac{\alpha!
(N-1)!}{\Gamma(N+s)} \cdot \frac{1}{(s+1)^{\gamma+1}}  \]
and
\[ \|f\|_{A^2_{\gamma}(B_N)}^2 = \sum_{s=0}^{\infty} \sum_{|\alpha|=s} |c_{\alpha}|^2
\frac{ \alpha!
\Gamma(N+\gamma+1)}{\Gamma(N+s+\gamma+1)} . \]
This last formula follows from the fact that 
\[ \|z^{\alpha} \|_{A^2_{\gamma}(B_K)}^2 = \frac{\alpha!\Gamma(N+\gamma+1)}{\Gamma(N+|\alpha|+\gamma+1)} \]
(see Lemma~1.11 in \cite{Zhubook}.) 

The result in (a) will follow if we can show that
$$ \left( \frac{\alpha! (N-1)!}{\Gamma(N+s)} \cdot \frac{1}{(s+1)^{\gamma+1}} \right)
\cdot \left( \frac{\Gamma(N+s+\gamma+1)}{
\alpha! \Gamma(N+\gamma+1)} \right) $$
is bounded above and below by positive constants, depending only
on $N$ and $\gamma$, for all $s\geq 0$.
This follows easily from the fact that, by Stirling's formula,
$$\lim_{s\rightarrow \infty}\frac{\Gamma(N+s+\gamma+1)}{(s+1)^{\gamma+1}\Gamma(N+s)}=1.$$

From (a) we know that $A^2_{N-K+\gamma}(B_K) = H^2(\widetilde
\beta, B_K)$ where $\widetilde
\beta(s)^2 = (s+1)^{-(N-K+\gamma+1)}$.
Thus it suffices to show that
$$\left[\frac{(N-1)!(K-1+s)!}{(K-1)!(N-1+s)!}(s+1)^{-(\gamma+1)}\right]\cdot\left[(s+1)^{(N-K+\gamma+1)}\right]$$
is bounded
above and below by positive constants (depending on $N$ and $K$)
for all $s\geq 0$.  Straightforward estimates show that
$$\frac{(N-1)!}{(K-1)!}\geq \frac{(N-1)!(K-1+s)!}{(K-1)!(N-1+s)!}\cdot (s+1)^{N-K}\geq
\frac{(N-1)!}{(K-1)!} \left(\frac{1}{N+1}\right)^{N-K}$$
for all $s\geq 0$, and the desired result follows.

\end{proof}

Since the proof of Theorem~\ref{scvmain} ultimately relies on an inductive argument,
we will work with certain restriction and extension operators on weighted Hardy spaces.
These are defined next.

Let $K, N \in \mathbb{N}$ with $1 \leq K < N$. Given a sequence $\{ \beta(s) \}$
of positive numbers, we define the associated sequence $\{ \widetilde \beta(s) \}$ by
\[ \widetilde \beta(s)^2 = \frac{(N-1)!(K-1+s)!}{(K-1)!(N-1+s)!} \beta(s)^2. \]
We can then define the extension operator $E: H^2(\widetilde \beta, B_K) \rightarrow H^2(\beta, B_N)$ by 
\[ (Ef)(z_1, \dots, z_N) = f(z_1, \dots, z_K), \] 
and the restriction operator $R: H^2(\beta, B_N)\rightarrow H^2(\widetilde \beta, B_K)$ by
\[ (Rf)(z_1, \dots, z_K) = f(z_1, \dots, z_K, 0'). \]
The next result establishes properties of these operators;
it is an extension of Proposition 2.21 in \cite{cowenmac} which applies to the case $K=1$.
\begin{lemma}\label{restrictextend}
\begin{itemize}
\item[(a)] The extension operator $E$ is an isometry of $H^2(\widetilde \beta, B_K)$ into $H^2(\beta, B_N)$.

\item[(b)] The restriction operator $R$ is a norm-decreasing map of $H^2(\beta, B_N)$ onto $H^2(\widetilde \beta, B_K)$. 

\end{itemize}
\end{lemma}

\begin{proof}

For (a), let $f \in H^2(\widetilde \beta, B_K)$ with homogeneous expansion $f= \sum f_s$ 
with $f_s$ as in Equation~(\ref{homogeneousfnc}) for $z\in {\mathbb C}^K$.
Then $Ef = \sum f_s$ also, and writing $\|f_s\|_{2,K}$ for the norm
of $f_s$ in $L^2(\partial B_K,\sigma)$ we have
\begin{align*}
\| f \|^2_{H^2(\widetilde \beta, B_K)} & = \sum_s \|f_s\|^2_{2,K} \widetilde \beta(s)^2 \\
 & = \sum_s \sum_{|\alpha|=s} |c_{\alpha}|^2 \frac{(K-1)! \alpha!}{(K-1+s)!} \widetilde \beta(s)^2 \\
 & = \sum_s \sum_{|\alpha|=s} |c_{\alpha}|^2 \frac{(K-1)! \alpha!}{(K-1+s)!} \cdot \frac{(N-1)!(K-1+s)!}{(K-1)!(N-1+s)!} \beta(s)^2 \\
 & = \sum_s \sum_{|\alpha|=s} |c_{\alpha}|^2 \frac{(N-1)! \alpha!}{(N-1+s)!} \beta(s)^2 \\
 & = \sum_s \|f_s\|^2_{2,N} \beta(s)^2 \\
 & = \| Ef \|^2_{H^2(\beta, B_N)}. 
\end{align*}
Therefore, $E$ is an isometry.

For (b) let $f \in H^2(\beta, B_N)$ have homogeneous expansion
$f=\sum f_s$ with $f_s$ as in Equation~(\ref{homogeneousfnc}) for $z\in B_N$.
For each nonnegative integer $s$, let $A_s = \{ \alpha: \alpha = (\alpha', 0') \}$, where $\alpha'$ is a multi-index with $K$ entries and $0'$ denotes the zero vector in $\mathbb{C}^{N-K}$, and let $B_s$ consist of all other multi-indices $\alpha$ with $N$ entries satisfying $|\alpha|=s$. Writing
\[ f_s(z) = \sum_{A_s} c_{\alpha} z^{\alpha} + \sum_{B_s} c_{\alpha} z^{\alpha}, \]
it follows that
\[ (Rf)(z_1, \dots, z_K) = \sum_s \sum_{A_s} c_{\alpha} z_1^{\alpha_1} \cdots z_K^{\alpha_K}, \]
and 
\begin{align*}
\|Rf\|^2_{H^2(\widetilde \beta, B_K)} & = \sum_s \sum_{A_s} |c_{\alpha}|^2 \frac{(K-1)! \alpha!}{(K-1+s)!} \widetilde \beta(s)^2 \\
 & = \sum_s \sum_{A_s} |c_{\alpha}|^2 \frac{(N-1)! \alpha!}{(N-1+s)!} \beta(s)^2 \\
 & \leq  \sum_s \sum_{|\alpha|=s} |c_{\alpha}|^2 \frac{(N-1)! \alpha!}{(N-1+s)!} \beta(s)^2 \\
 & = \| f \|^2_{H^2(\beta, B_N)}. 
\end{align*}
Therefore, $R$ is norm-decreasing. To see that $R$ is surjective, let $f \in H^2(\widetilde \beta, B_K)$ with
\[ f(z_1, \dots, z_K) = \sum_s \sum_{|\alpha'|=s} c_{\alpha'} z_1^{\alpha_1'} \cdots z_K^{\alpha_K'}, \]
where $\alpha'$ is a multi-index with $K$ entries. 
Then $f=RF$, where
\[ F(z) = \sum_s \sum_{A_s} c_{\alpha} z^{\alpha} . \]
and $c_{\alpha} \equiv c_{\alpha'}$ for $\alpha = (\alpha', 0')$ as before. Also,
\begin{align*} 
\|F\|^2_{H^2(\beta, B_N)} & = \sum_s \sum_{A_s} |c_{\alpha}|^2 \frac{(N-1)! \alpha!}{(N-1+s)!} \beta(s)^2 \\
 & = \sum_s  \sum_{A_s} |c_{\alpha}|^2 \frac{(K-1)! \alpha!}{(K-1+s)!} \widetilde \beta(s)^2 \\
 & = \sum_s  \sum_{|\alpha'|=s} |c_{\alpha'}|^2 \|z_1^{\alpha_1'} \cdots z_K^{\alpha_K'}\|^2_{2,K} \\ 
 & = \|f\|^2_{H^2(\widetilde \beta, B_K)} < \infty, 
\end{align*}
so $F \in H^2(\beta, B_N)$.

\end{proof}

Recall that by an \textit{affine subset of dimension $k$} in $B_N$, we mean the intersection
of $B_N$ with a translate of a $k$-dimensional subspace of ${\mathbb C}^N$.

\begin{proof}[Proof of Theorem~\ref{scvmain}]
Suppose that $\varphi$ and $\psi$ are linear-fractional maps with
$\|\varphi\|_{\infty}=\|\psi\|_{\infty}=1$. We will show that if 
$C_{\varphi}-C_{\psi}$ is compact on ${\mathcal H}$, then $\varphi=\psi$.
The hypothesis $\|\varphi\|_{\infty}=1$ implies that there exists a point
$\zeta$ 
in $\partial B_N$ with $|\varphi(\zeta)|=1$. Composing on the left and right
by unitaries, there is no loss of generality in assuming $\zeta=e_1$ and
$\varphi(e_1)=e_1$. By
Theorem~\ref{samefirst}, we must have
$\psi(e_1)=e_1$ as well. 

The argument proceeds exactly as in the proof of Theorem~\ref{univalentcase}
up to the point where the relationship
\begin{equation}\label{tauxi2}
\varphi\circ\sigma_{\varphi}=\psi\circ\sigma_{\varphi},
\end{equation}
is obtained.  Since Theorem~\ref{univalentcase} covers the case that
at least one of $\varphi$ and $\psi$ is one-to-one, we now
only consider the case that neither is univalent.
This implies that there is a smallest $k_1$ with $1\leq k_1<N$ so that $\varphi(B_N)$ 
is contained in an affine set of dimension $k_1$, and
there is a smallest $k_2$ with $1\leq k_2<N$ so that $\psi(B_N)$ is contained
in an affine set of dimension $k_2$.  
Since the roles of $\varphi$ and $\psi$ can be reversed,
there is no loss of generality in assuming $k_1\geq k_2$.

Our first goal is to show that equality
$k_1=k_2$
holds as a consequence of Equation~(\ref{tauxi2}).  
By Proposition 13 in \cite{cowenmac}, $\sigma_{\varphi}(B_N)$ is
also contained in an $k_1$-dimensional affine set.
Set $\sigma_{\varphi}(0)=p$, and let $\phi_p$ be
an automorphism of $B_N$ sending $p$ to $0$ and
satisfying $\phi_p=\phi_p^{-1}$.  
Since $\phi_p\circ\sigma_{\varphi}$ fixes the origin,
and maps the ball into a $k_1$-dimensional affine set,
we may write, as in \cite{cowenmac},
$$\phi_p\circ\sigma_{\varphi}=L\circ\tau$$
where $L$ is linear of rank $k_1$ and $\tau$ is a one-to-one linear fractional
map.  Specifically, if 
$$\phi_p\circ\sigma_{\varphi}=\frac{Az}{\langle z,C\rangle +1}$$
we can choose 
$$L(z)=Az\mbox{   and   }\tau(z)=\frac{z}{\langle z,C\rangle +1}.$$
(Note that $L$ and $\tau$ need not be self-maps of $B_N$, though
their composition is, and
both are defined and analytic on a neighborhood of the closed
ball).  From this it follows that 
$\sigma_{\varphi}=\phi_p\circ L\circ\tau$.  Taking Krein adjoints
on both sides we have
$\varphi=\sigma_{{\tau}}\circ L^*\circ\sigma_{\phi_p}$,
where $\sigma_{\tau}(z)=z-C$.
We have $\sigma_{\phi_p}=\phi_p^{-1}=\phi_p$, so that
\begin{equation}\label{dimcomp}
\varphi\circ\sigma_{\varphi}=\sigma_{\tau}\circ L^*\circ\sigma_{\phi_p}\circ\phi_p\circ L\circ\tau=
\sigma_{\tau}\circ L^*L\circ\tau,
\end{equation} where $L^*L$ is linear
with rank $k_1$, $\tau$ is univalent,
and $\sigma_{\tau}$ is a translation.  Thus the image of the ball $B_N$ under
$\varphi\circ\sigma_{\varphi}$ cannot be contained in a $k$ dimensional
affine set for any $k<k_1$, and the relation $\varphi(\sigma_{\varphi}(B_N))=\psi(\sigma_{\varphi}(B_N))
\subseteq\psi(B_N)$ says that strict inequality $k_1>k_2$ is impossible, and therefore $k_1=k_2$ as desired.
We denote the common value of $k_1$ and $k_2$ by $K$.  

Thus $\varphi(B_N)$ is contained in a $K$-dimensional
affine set $A_1$ and is not contained in any 
affine set of smaller dimension, and $\psi(B_N)$ is contained in a $K$-dimensional
affine set $A_2$, and is not contained in any affine set of
smaller dimension.  We have 
$$A_2\supseteq \psi(B_N)\supseteq \psi(\sigma_{\varphi}(B_N))=\varphi(\sigma_{\varphi}(B_N))$$
where $A_1\supseteq\varphi(\sigma_{\varphi}(B_N))=\sigma_{\tau}L^*L\tau(B_N)$
for linear $L^*L$ of rank $K$ and univalent linear fractional $\tau$ and $\sigma_{\tau}$.
This forces $A_1=A_2$; that is, the range of $\varphi$ and the range
of $\psi$ are contained in the same $K$-dimensional affine set, which we will simply denote $A$. Note
that $e_1 \in \overline{A}$. 

Our goal is to show that $\varphi=\psi$. Let $\zeta\in\partial B_N$
with $\zeta\neq e_1$.  Let $\Lambda_{\zeta}$ be a $K$-dimensional affine subset of
$B_N$, containing $e_1$ and
$\zeta$ in its boundary, whose intersection with $B_N$ is a $K$-dimensional ball.
We will write $\widetilde{B_K}$ for 
$\{(z_1,z_2,\ldots,z_K,0'')\equiv(z',0'')\in B_N\}$, where
$0''$ denotes the $0$ in ${\mathbb C}^{N-K}$.  Let $\rho_1$ be an automorphism
of $B_N$ fixing $e_1$ with $\rho_1(\widetilde{B_K})=\Lambda_{\zeta}$ and let $\rho_2$ be an automorphism
of $B_N$ fixing $e_1$ with $\rho_2(A)=\widetilde{B_K}$; such automorphisms exist because of the two-fold
transitivity of the automorphisms on $\partial B_N$.
Note that $\rho_2\circ\varphi\circ\rho_1$ and $\rho_2\circ\psi\circ\rho_1$ are linear-fractional
self-maps of $B_N$ with $\rho_2\circ\varphi\circ\rho_1(\widetilde{B_K})\subseteq \widetilde{B_K}$
and $\rho_2\circ\psi\circ\rho_1(\widetilde{B_K})\subseteq \widetilde{B_K}$.  
Let $\pi$ be the projection of ${\mathbb C}^N$ onto ${\mathbb C}^K$ defined
by $\pi(z',z'')=z'$ and define maps $\mu$ and $\nu$ on $B_K$ 
by $$\mu(z')=\pi\circ\rho_2\circ\varphi\circ\rho_1(z',0'')$$
and
$$\nu(z')=\pi\circ\rho_2\circ\psi\circ\rho_1(z',0'').$$
These are linear-fractional self-maps of
$B_K$ fixing $(1,0\ldots,0)\in\partial B_K$.  

Write ${\mathcal H}$ as $H^2(\beta, B_N)$ with $\beta(s)^2=(s+1)^{-(\gamma+1)}$,
where $\gamma=-1$ if ${\mathcal H}=H^2(B_N)$
and $\gamma=\alpha$ if
${\mathcal H}=A^2_{\alpha}(B_N)$, up to an 
equivalent norm.
We claim that $C_{\mu}-C_{\nu}$ is compact on the weighted Hardy space
$H^2(\widetilde{\beta},B_K)$ where 
$$\widetilde{\beta}(s)^2=\frac{(N-1)!(K-1+s)!}{(K-1)!(N-1+s)!}\beta(s)^2.  $$
Since, by
Lemma~\ref{lem:equiv1}, $H^2(\widetilde{\beta},B_K)=A^2_{N-K+\gamma}(B_K)$ with equivalent norms, it will follow
that $C_{\mu}-C_{\nu}$ is compact on $A^2_{N-K+\gamma}(B_K)$.  

To prove the claim it suffices to show that if $\{f_n\}$ is a bounded sequence
on $H^2(\widetilde{\beta},B_K)$ and
$f_n\rightarrow 0$ almost uniformly on $B_K$, then 
$$\|(C_{\mu}-C_{\nu})f_n\|_{H^2(\widetilde{\beta},B_K)}\rightarrow 0.$$
Let $f_n$ be such a sequence in $H^2(\widetilde{\beta},B_K)$.  Define functions
$F_n$ on $B_N$ by
$$F_n=Ef_n,$$
where $E: H^2(\widetilde{\beta}, B_K) \rightarrow H^2(\beta,B_N)$ is the extension operator
defined by 
\[ (Ef)(z_1, \dots, z_N) = f(z_1, \dots, z_K). \]
By Lemma~\ref{restrictextend}, $E$ is an isometry, so the functions $F_n$ form a
bounded sequence in ${\mathcal H}$
and $F_n\rightarrow 0$ uniformly on compact subsets of $B_N$.  Since
$C_{\rho_2\circ\varphi\circ\rho_1}-C_{\rho_2\circ\psi\circ\rho_1}=C_{\rho_1}(C_{\varphi}-C_{\psi})C_{\rho_2}$ is
compact on ${\mathcal H}$, we have
$$\|F_n\circ\rho_2\circ\varphi\circ\rho_1-F_n\circ\rho_2\circ\psi\circ\rho_1\|_{H^2(\beta,B_N)} \rightarrow 0$$
Define the restriction operator $R: H^2(\beta,B_N)\rightarrow H^2(\widetilde{\beta}, B_K)$ by
\[ (Rf)(z_1, \dots, z_K) = f(z_1, \dots, z_K, 0''). \]  
By Lemma~\ref{restrictextend}, $R$ is a norm-decreasing map of $H^2(\beta,B_N)$ onto
$H^2(\widetilde{\beta}, B_K)$ and
so  
$$\|R(F_n\circ\rho_2\circ\varphi\circ\rho_1)-R(F_n\circ\rho_2\circ\psi\circ\rho_1)\|_{H^2(\widetilde{\beta}, B_K)}
\rightarrow 0$$
But $R(F_n\circ\rho_2\circ\varphi\circ\rho_1)=f_n\circ\mu$ on $B_K$ and
$R(F_n\circ\rho_2\circ\psi\circ\rho_1)=f_n\circ\nu$ on $B_K$,
so the claim is verified, and $C_{\mu}-C_{\nu}$ is compact on $A^2_{ N-K+\gamma}(B_K)$.
Since $K<N$ and $\mu$ and $\nu$ are linear-fractional self-maps of $B_K$
with $\|\mu\|_{\infty}=\|\nu\|_{\infty}=1$, by induction this forces
$\mu = \nu$, which in turn says that $\varphi=\psi$ on the affine set
$\Lambda_{\zeta}$ containing $\zeta$ and $e_1$. 
Since $\zeta $ is an arbitrary
point in $\partial B_N$, this says $\varphi=\psi$ in $B_N$.

\end{proof}

\end{document}